\documentclass{article}

\PassOptionsToPackage{numbers, compress}{natbib}
\usepackage{natbib}
\usepackage{PRIMEarxiv}

\usepackage[utf8]{inputenc} % allow utf-8 input
\usepackage[T1]{fontenc}    % use 8-bit T1 fonts
\usepackage{hyperref}       % hyperlinks
\hypersetup{hidelinks}
\usepackage{url}            % simple URL typesetting
\usepackage{booktabs}       % professional-quality tables
\usepackage{amsfonts}       % blackboard math symbols
\usepackage{nicefrac}       % compact symbols for 1/2, etc.
\usepackage{microtype}      % microtypography
\usepackage{xcolor}         % colors

\usepackage{amsthm}
\usepackage{amssymb}
\usepackage{amsmath}
\newtheorem{thm}{Theorem}
\newtheorem{prop}[thm]{Proposition}

\newtheorem{rem}[thm]{Remark}

\usepackage{caption}
\usepackage{subcaption}
\renewcommand{\tilde}{\widetilde}
\DeclareMathOperator{\wce}{wce}
\usepackage{bbm}
\newcommand{\bm}[1]{{\mbox{\boldmath $#1$}}}
\DeclareMathOperator{\dist}{dist}
\DeclareMathOperator{\cv}{conv}
\newcommand{\p}[1]{\mathbb{P}\!\left(#1\right)}
\renewcommand{\P}[1]{\mathbb{P}\!\left(#1\right)}

\usepackage{comment}
\newcommand{\Hil}{\mathcal{H}}
\newcommand{\E}[1]{\mathbb{E}\!\left[#1\right]}

\newcommand{\ord}[1]{\mathcal{O}\!\left(#1\right)}
\newcommand{\R}{\mathbb{R}}
\newcommand{\dd}{\mathrm{d}}
\newcommand{\ve}{\varepsilon}
\newcommand{\X}{\mathcal{X}}
\newcommand{\ip}[1]{\left\langle #1 \right\rangle}
\newcommand{\nor}[1]{\left\lVert #1 \right\rVert}

\renewcommand{\paragraph}[1]{
    \noindent\textbf{#1}\ 
}

\usepackage{algorithm}
\usepackage{algpseudocode}
\usepackage{graphicx}

\title{Convex-Geometric Error Bounds for Positive-Weight Kernel Quadrature}
\newcommand{\RUNTITLE}{
    Convex-Geometric Error Bounds for Positive-Weight Kernel Quadrature
}
\author{%
  Satoshi Hayakawa \\
  The University of Tokyo \\
  \texttt{hayakawa@mist.i.u-tokyo.ac.jp} \\
}

\begin{document}
\maketitle

\begin{abstract}
Kernel quadrature can exploit RKHS spectral structure and outperform Monte
Carlo on smooth integrands, but optimized quadrature weights are generally
signed and may be numerically unstable. We study whether spectral acceleration
remains possible when the weights are constrained to be positive, i.e., simplex
weights. In the exact-target fixed-pool setting, an evaluated i.i.d. candidate
pool of size $N$ is already available and the task is to reweight it so as to
approximate the kernel mean embedding. We show that this positive reweighting
problem is governed not by the equal-weight empirical average, but by the random
convex hull generated by the pool. Our main geometric result shows that the mean
of a bounded $d$-dimensional random vector can be approximated by a convex
combination of $N$ i.i.d. samples at accuracy $\mathcal{O}(d/N)$ with high probability,
sharper than equal-weight averaging in the fixed-dimensional regime. We transfer
this $d$-dimensional convex-hull approximation to full RKHS worst-case error
through an augmented Mercer-truncation argument. The resulting positive-weight
KQ bounds consist of a spectral tail term and a finite-sample convex-hull term,
yielding Monte-Carlo-beating rates in favorable spectral regimes, including
near-$\mathcal{O}(1/N)$ rates up to logarithmic factors under exponential spectral
decay. We also provide a constructive Frank--Wolfe algorithm that operates
directly on the pool atoms, maintains simplex weights, and admits an explicit
optimization-error bound.
\end{abstract}

\section{Introduction}

Kernel quadrature (KQ) is a kernel-based approach to numerical integration,
closely related to Bayesian quadrature (BQ) and probabilistic
integration~\citep{oha91,ras03,bri19}. For sufficiently regular integrands,
KQ can exploit spectral structure in a reproducing kernel Hilbert space (RKHS) that is invisible to plain Monte
Carlo and thereby converge faster than the usual
$\mathcal{O}(N^{-1/2})$ rate in the number of points~\citep{bac17,kan20}.
Unconstrained kernel-based rules, however, may produce numerically unstable
weights, motivating longstanding interest in positively weighted
rules~\citep{che10,hayakawa21b,KKS19,van20}. In this paper, positive weights
mean nonnegative weights that sum to one, i.e., simplex or
convex-combination weights.

Whether positive-weight KQ can systematically improve over Monte Carlo is a
subtle question. Kernel herding and related constructions suggested fast
rates under favorable assumptions~\citep{che10}, but the conditional-gradient
viewpoint of \citet{bac12} clarified that the strongest such assumptions are
not generally available in infinite-dimensional RKHSs. Subsequent
herding-type analyses in broad RKHS settings have therefore mostly remained
at the Monte-Carlo scale, except under additional structure or modified
algorithms such as sparse herding variants~\citep{lacostejulien15,tsu22,tsu26}.
Beyond herding, subsampling-based positive KQ methods such as
thinning~\citep{dwi24,dwi22} and recombination~\citep{hayakawa21b,hayakawa2023sampling}
have obtained rates beyond Monte Carlo, but a general mechanism for such
improvement in the simple i.i.d. fixed-pool setting has remained unclear.

By contrast, without the positivity constraint, sharper spectral rates are
known for optimized KQ weights. Random-feature and Nystr\"om viewpoints,
including the i.i.d. and leverage-score analyses of \citet{bac17} and
\citet{chatalic25}, give unrestricted-weight rates controlled by spectral
decay or effective dimension. Our question is whether a comparable
phenomenon survives under the positive-weight constraint. We study this in
an exact-target fixed-pool setting: an evaluated i.i.d. candidate pool is
already available, and the task is to reweight it using positive weights.
Without this constraint, the same fixed-node kernel mean approximation
problem underlies BQ, where posterior variance is minimized over
unrestricted weights~\citep{hus12,bri15}. Our focus is the positive-weight
counterpart, namely approximation of the kernel mean embedding by a point in
the convex hull of $\{k(x_i,\cdot)\}_{i=1}^N$. We assume that the target
kernel mean can be evaluated on the pool points, as in empirical targets or
analytically tractable kernels.

Our main observation is that positive reweighting is governed not by the
equal-weight empirical average, but by the random convex hull generated by
the pool. Building on relaxed Tukey-depth/random-convex-hull machinery
\citep{hayakawa21a}, rooted in Tukey or halfspace depth~\citep{tuk75,rou99},
we show that the mean of a bounded $d$-dimensional random vector can be
approximated by a convex combination of $N$ i.i.d. samples at accuracy
$\ord{d/N}$ with high probability. We then transfer this finite-dimensional
convex-hull approximation to KQ using a residual-diagonal augmentation of a
truncated Mercer representation, following the spirit of earlier
positive-weight KQ analyses~\citep{hayakawa21b}. This yields full
worst-case-error bounds governed by a spectral tail term and a finite-sample
convex-hull term. In exponentially decaying spectral regimes, the resulting
positive-weight rates are near-$N^{-1}$ up to logarithmic factors,
comparable in scale to unrestricted least-squares/Nystr\"om rates but under
the additional positive-weight constraint. We also give a constructive
counterpart: Frank--Wolfe on the convex hull of the pool atoms approximately
solves the positive-weight optimum on the fixed pool and admits an explicit
optimization-error bound~\citep{jaggi13,bac12}. The Mercer expansion and
truncation level are used only in the analysis; the algorithm itself only
needs the evaluated pool and the target kernel mean on the pool.

\paragraph{Contributions.}
Our main contributions are as follows.
\begin{enumerate}
    \item We prove a convex-geometric, high-probability approximation result
    for bounded random vectors: the population mean can be approximated
    by a convex combination of $N$ i.i.d. samples at accuracy $\ord{d/N}$
    in dimension $d$ (Section~\ref{sec:convex-hull}).

    \item We transfer this random convex-hull approximation result to
    positive-weight kernel quadrature in Section~\ref{sec:kq-existence},
    via an augmented truncation of the Mercer expansion, yielding full RKHS
    worst-case error bounds of the form
    $\wce(Q_N;\mathcal{H},\mu)
    =
    \mathcal{O}\bigl((\sum_{j>d}\sigma_j)^{1/2} + d/N\bigr)$,
    where $(\sigma_j)_{j\ge1}$ are the Mercer eigenvalues of the kernel.

    \item In Section~\ref{sec:kq-construction}, we give a constructive
    counterpart by running Frank--Wolfe directly in the RKHS on the convex
    hull of the pool atoms, obtaining an explicit optimization-error bound
    for a positive-weight fixed-pool algorithm.
\end{enumerate}

\section{Related Work}\label{sec:related-work}

\paragraph{Positive-weight KQ and measure reduction.}
Positive-weight KQ is closely tied to herding and conditional-gradient
methods. Kernel herding was introduced as a positive-weight construction
with the promise of fast deterministic integration~\citep{wel09,che10}, and
\citet{bac12} later identified its conditional-gradient structure and
explained why the strongest fast-rate assumptions are restrictive in general
infinite-dimensional RKHSs. Subsequent broad RKHS guarantees for
herding-type methods have largely remained at the Monte-Carlo scale, apart
from more specialized variants~\citep{lacostejulien15,tsu22,tsu26}. Other
positive-weight measure-reduction methods include support points~\citep{mak18},
kernel thinning and generalized kernel thinning~\citep{dwi24,dwi22},
positive-weight KQ via subsampling~\citep{hayakawa21b}, and positive rules
on arbitrary sample sets~\citep{van20}. These methods show that positivity
can be compatible with compression and, in some cases, rates beyond Monte
Carlo, but they typically rely on point selection, thinning, recombination,
or subsampling rather than positive reweighting of a fixed i.i.d. pool.

\paragraph{Unrestricted weights and Bayesian quadrature.}
A separate line of work obtains spectral acceleration with optimized weights
that are generally signed or otherwise unrestricted. This includes
randomized or optimized sampling for KQ~\citep{bac17,bri17}, Nystr\"om kernel
mean embeddings and Nystr\"om-based KQ~\citep{cha22,hayakawa2023sampling},
leverage-score sampling with optimized least-squares weights~\citep{chatalic25},
determinantal point processes and continuous volume sampling~\citep{belhadji2019kernel,belhadji2020kernel},
and randomly pivoted Cholesky~\citep{epperly2023kernel}. The results of
\citet{bac17} and \citet{chatalic25} are especially relevant as i.i.d. or
fixed-node analogues with rates expressed through eigenvalue decay or
effective dimension, but their optimized weights are not constrained to be
positive or to sum to one. BQ fits the same unrestricted-weight fixed-node
perspective once the nodes are fixed: posterior variance minimization is
kernel mean approximation over unrestricted weights~\citep{hus12,bri15},
while adaptive and batch BQ methods acquire nodes using posterior
uncertainty or recombination ideas~\citep{gunter2014sampling,kanagawa2019convergence,ada22}.
Our contribution is complementary: we do not improve these unrestricted
rates, but show that comparable spectral acceleration can survive under the positivity constraint; for instance, it reaches the near-$N^{-1}$ scale under exponential eigenvalue decay.

\paragraph{Random convex hulls and sparse convex approximation.}
The geometric input of the paper is closest to work on random convex hulls,
depth methods, and recombination-based cubature. We build on the relaxed
depth/random-convex-hull machinery of \citet{hayakawa21a}, rooted in Tukey
or halfspace depth~\citep{tuk75,rou99}. Related ideas appear in
recombination and cubature constructions~\citep{lit12,lyo04}, including
randomized multivariate cubature analyses based on random convex hulls and
hypercontractivity~\citep{hayakawa2023hypercontractivity}. Approximate
Carath\'eodory theorems are also relevant because they study sparse convex
combinations with explicit approximation guarantees~\citep{jaggi13,mirrokni17tight,combettes2023revisiting}.
Most random-convex-hull bounds, however, concern exact inclusion or
finite-dimensional approximation and do not by themselves yield RKHS
worst-case-error guarantees for KQ. We use random convex hulls as a
probabilistic approximation tool for the kernel mean itself: an $\ord{d/N}$
high-probability mean approximation bound is transferred to positive-weight
RKHS kernel mean approximation, giving fixed-pool KQ guarantees beyond the
finite-dimensional or problem-specific cubature settings where such
arguments had previously been used.

\section{Setup}\label{sec:setup}

\paragraph{Kernel quadrature and mean embedding.}
Let $\mu$ be a probability measure on $\mathcal{X}$, and let $k:\mathcal{X}\times\mathcal{X}\to\mathbb{R}$ be a positive-definite kernel with a real-valued reproducing kernel Hilbert space (RKHS) $\Hil$. Our goal is to approximate $I_\mu(f):=\int f(x)\,\dd\mu(x)$ from an evaluated pool $(x_i,f(x_i))_{i=1}^N$ with $x_i\sim_{\text{iid}}\mu$, under the assumption that further function evaluations are expensive or unavailable. We are interested in positive-weight rules, i.e.\ quadrature rules of the form $Q_N(f):=\sum_{i=1}^N w_i f(x_i)$ with weight vector $\bm{w}=(w_i)_{i=1}^N$ in the simplex $\Delta_N:=\{\bm{w}\in\R^N_{\ge0}:\sum_{i=1}^N w_i=1\}$.
Its evaluation criterion is the worst-case error over the unit ball of RKHS: $\wce(Q_N;\Hil, \mu):=
\sup_{\|f\|_\Hil\le 1}|I_\mu(f)-Q_N(f)|$.
Define the kernel mean embedding by $m_\mu:=\int k(x,\cdot)\,\dd\mu(x)\in\Hil$. By the reproducing property,
$ I_\mu(f)-Q_N(f)=\left\langle f,\,m_\mu-\sum_{i=1}^N w_i k(x_i,\cdot)\right\rangle_\Hil. $
Hence the worst-case error is
\begin{equation}
\wce(Q_N;\Hil, \mu)
:=
\sup_{\|f\|_\Hil\le 1}|I_\mu(f)-Q_N(f)|
=
\left\|m_\mu-\sum_{i=1}^N w_i k(x_i,\cdot)\right\|_\Hil,
\label{eq:wce-norm-form}
\end{equation}
so it reduces to approximating $m_\mu$ by a discrete measure supported on the pool $(x_i)_{i=1}^N$.

\paragraph{Spectral representation assumptions on the kernel.}
We assume that $k$ admits the Mercer expansion $k(x,y)=\sum_{j=1}^\infty \sigma_j e_j(x)e_j(y)$ with respect to $\mu$, where $(e_j)_{j\ge1}$ is an orthonormal set in $L^2(\mu)$ and $\sigma_1\ge \sigma_2\ge \cdots \ge 0$. Under classical continuity and compactness assumptions this is the usual Mercer theorem, and substantially more general measurable-domain versions are available as well~\citep{min06,ste12}. Throughout the main text we also assume the bounded-kernel regime $\sup_{x\in\mathcal X}\sqrt{k(x,x)}\le \kappa$ for some $\kappa>0$.
Our theoretical bounds in Section~\ref{sec:kq} are given in terms of the spectral decay $(\sigma_j)_{j\ge1}$ and the bounding constant $\kappa$.
Appendix~\ref{sec:bounded-reduction} gives a normalization viewpoint for extending the bounded-diagonal
assumption, though the resulting positive weights need not remain simplex weights.

\paragraph{What the algorithm needs.}
It is worth separating the information required by the analysis from the information required by the algorithm. The truncated Mercer expansion and truncation level $d$ enter the theory, where they are used to state and prove the error bounds in Section~\ref{sec:kq}, but they are \emph{not} inputs to the constructive procedure itself. Algorithmically, our methods (i.e., fixed-pool CQP/Frank--Wolfe) tested in Section~\ref{sec:experiments} only require an evaluated pool $x_1,\dots,x_N$, kernel evaluations on that pool (or the Gram matrix $(k(x_i, x_j))_{i,j=1}^N$), and access to the target kernel mean on the pool, namely $m_\mu(x_i)=\int k(x_i,x)\,\dd\mu(x)$ for $i=1,\dots,N$. In particular, the method does not require explicit knowledge of the Mercer eigenfunctions, eigenvalues, or a user-specified truncation level $d$.

\paragraph{Fixed-pool and exact-target setting.}
In the main text we study fixed-pool KQ: an evaluated pool $(x_i,f(x_i))_{i=1}^N$ is given, and the goal is to construct a stable positive-weight KQ supported on that pool. We further work in the exact-target setting as stated above, in which the target kernel mean embedding $m_\mu$ is available on the pool points. This includes, in particular, discrete target measures such as empirical measures supported on many unlabeled samples.
See also \eqref{eq:triangle-inequality} for a more practical extension.

\section{Random Convex-Hull Approximation of Mean Vectors}\label{sec:convex-hull}

We now state the probabilistic-geometric result that underlies our kernel quadrature application.
The key point is that the mean of a bounded random vector can be approximated by a convex combination of i.i.d.\ samples at accuracy $\ord{d/N}$ with high probability. This will later be transferred to positive kernel quadrature through an augmented Mercer representation.

\subsection{A one-sided probabilistic inequality}

Our starting point is the following one-sided inequality.

\begin{thm}\label{thm:one-sided}
Let $W$ be a centered random variable such that $\E{|W|^3}\le M\cdot\E{W^2}<\infty$ for some $M>0$. Then, for any positive integer $n$ and independent copies $W_1,\dots,W_n$ of $W$,
\begin{equation}
    \p{\frac1n\sum_{i=1}^n W_i \le \frac{2M}{n}} \ge \frac12.
    \label{eq:one-sided}
\end{equation}
In particular, the assumption is automatically satisfied if $|W|\le M$ almost surely.
\end{thm}

\begin{proof}[Proof sketch]
If the variance is small, i.e., if $\E{W^2} \le 2M^2/n$,
then \eqref{eq:one-sided} follows from Chebyshev's inequality.
The complementary regime follows from the Berry--Esseen theorem~\citep{ber41,ess42,kor12} after normalization. The full proof is deferred to Appendix~\ref{app:one-sided}.
\end{proof}

%\paragraph{Why a one-sided inequality?}
The threshold in Theorem~\ref{thm:one-sided} is of order $1/n$, whereas generic concentration inequalities are typically most informative at order $n^{-1/2}$. Thus, standard two-sided concentration bounds do not directly provide a useful lower bound of constant order for the one-sided event considered here.
The point of Theorem~\ref{thm:one-sided} is therefore not to improve generic tail bounds uniformly, but rather to furnish exactly the one-sided estimate that leads to the $O(d/N)$ convex-hull approximation rate in Theorem~\ref{thm:main-convex-hull}.

\subsection{Random convex-hull approximation of the mean}

To pass from one-dimensional probabilistic inequality to convex-hull approximation, we use the concept of relaxed Tukey depth \citep{hayakawa21a}, building on the classical notion of Tukey or halfspace depth~\citep{tuk75,rou99}. Given a $d$-dimensional random vector $\bm{Y}$,
for $\bm{\theta}\in\R^d$ and $\varepsilon\ge0$, define
\[
\alpha_{\bm{Y}}^\varepsilon(\bm{\theta})
:=
\inf_{\bm{c}\in\R^d,\ |\bm{c}|=1}\p{\bm{c}^\top(\bm{Y}-\bm{\theta})\le \varepsilon}.
\]
It is previously proven \citep[Theorem 14]{hayakawa21a} that if $\alpha_{\bm{Y}}^\varepsilon(\bm{\theta})$ is bounded away from zero, we can quantify the distance between the random convex hull and $\bm{\theta}$ as follows:
\begin{prop}[{\citep[Theorem 14]{hayakawa21a}}]\label{prop:hayakawa-bound}
    Given $d$-dimensional i.i.d.~random vectors $\bm{Y}, \bm{Y}_1, \ldots, \bm{Y}_N$, let 
    $\bm\theta \in\R^d$ and $\ve\ge0$. If we have $N \ge 3d/\alpha_{\bm{Y}}^\ve(\bm{\theta})$, then we have
    \[
        \P{\dist\!\left(\bm\theta, \cv\{\bm{Y}_1,\ldots,\bm{Y}_N\}\right)\le \ve}
        >1-2^{-d},
    \]
    where, for $A\subset\R^d$,
    $\cv A:=\{\sum_{j=1}^m c_j\bm{a}_j\mid m\ge1,\ c_j\ge0,\ \sum_{j=1}^mc_j=1\}$ denotes the convex hull of $A$ and 
    $\dist(\bm\theta, \cv A):=\inf_{\bm{a}\in \cv\! A}\lvert\bm\theta-\bm{a}\rvert$ denotes the distance between $\bm\theta$ and $\cv A$.
\end{prop}

To bound the relaxed Tukey depth in our case,
applying Theorem~\ref{thm:one-sided} to all one-dimensional projections yields the following lower bound.

\begin{prop}\label{prop:depth-lower-bound}
Let $\bm{X}$ be a centered $d$-dimensional random vector such that $|\bm{X}|\le M$ almost surely, and let $\bm{S}_n:=n^{-1}\sum_{i=1}^n \bm{X}_i$ with $\bm{X}_1,\dots,\bm{X}_n$ i.i.d.\ copies of $\bm{X}$.
Then, we have
$
\alpha_{\bm{S}_n}^{2M/n}(\bm{0})\ge 1/2
$.
\end{prop}
\begin{proof}
For any unit vector $\bm{c}\in\R^d$, the scalar random variables $W_i:=\bm{c}^\top \bm{X}_i$ satisfy $|W_i|\le M$. Applying Theorem~\ref{thm:one-sided} to $n^{-1}\sum_i W_i=\bm{c}^\top \bm{S}_n$ gives the result.
\end{proof}

Combining Propositions~\ref{prop:hayakawa-bound}~\&~\ref{prop:depth-lower-bound}, we obtain the main approximation result.

\begin{thm}[Random convex-hull approximation of the mean]\label{thm:main-convex-hull}
Let $\bm{X}$ be a centered $d$-dimensional random vector such that $|\bm{X}|\le M$ almost surely. Then, for any positive integer $n$, if $N\ge 6dn$, the convex hull of $N$ i.i.d.\ copies $\bm{X}_1,\dots,\bm{X}_N$ of $\bm{X}$ satisfies
\begin{equation}
    \p{\dist\!\left(\bm{0},\cv\{\bm{X}_1,\dots,\bm{X}_N\}\right)\le \frac{2M}{n}}
    > 1-2^{-d}.
    \label{eq:cv-main}
\end{equation}
\end{thm}

\begin{proof}[Proof]
Let $\bm{S}_n^{(i)}:=n^{-1}\sum_{j=1}^n\bm{X}_{(i-1)n+j}$ be the block average for $i=1,\ldots, 6d$.
They are i.i.d. copies of the $d$-dimensional random vector $\bm{S}_n$ in Proposition~\ref{prop:depth-lower-bound}.
Since $6d\ge 3d/\alpha_{\bm{S}_n}^{2M/n}(\bm{0})$
from Proposition~\ref{prop:depth-lower-bound}, applying Proposition~\ref{prop:hayakawa-bound} yields
\begin{equation}
    \P{\dist(\bm{0}, \cv\{\bm{S}_n^{(1)},\ldots, \bm{S}_n^{(6d)}\})\le \frac{2M}n} > 1-2^{-d}.
    \label{eq:intermediate-convex-hull}
\end{equation}
Since $\bm{S}_n^{(i)}\in\cv\{\bm{X}_1,\ldots,\bm{X}_N\}$,
we also have
$\cv\{\bm{S}_n^{(1)},\ldots, \bm{S}_n^{(6d)}\}\subset\cv\{\bm{X}_1,\ldots,\bm{X}_N\}$.
From this inclusion and \eqref{eq:intermediate-convex-hull},
we can obtain the desired bound \eqref{eq:cv-main}.
\end{proof}
Note that, if $\bm{X}$ is not assumed centered but still satisfies $|\bm{X}|\le M$ almost surely, then with probability at least $1-2^{-d}$ there exists a convex combination $(c_i)_{i=1}^N\in\Delta_N$ such that
\[
\left\lvert
\sum_{i=1}^N c_i \bm{X}_i-\E{\bm{X}}
\right\rvert
\le \frac{4M}{\lfloor N/(6d)\rfloor}.
\]
Indeed, one applies Theorem~\ref{thm:main-convex-hull} to the centered random vector $\bm{X}-\E{\bm{X}}$, which is bounded in norm by $2M$. Thus, the key message in this section is that, once convex combinations are allowed for i.i.d.\ vectors, the mean vector can be approximated at scale $\ord{d/N}$ rather than the usual $\ord{N^{-1/2}}$ of equal-weight averaging.

\section{Consequences for Kernel Quadrature}
\label{sec:kq}

We now transfer the random convex-hull approximation result in the previous section to RKHS worst-case errors.
We prove that a convex combination with a favorable error bound \textit{exists} in Section~\ref{sec:kq-existence}
and discuss how to obtain a practical set of weights in Section~\ref{sec:kq-construction}.

\subsection{Augmented mean approximation for bounding worst-case errors}\label{sec:kq-existence}

Recall from Section~\ref{sec:setup} that we have the Mercer expansion
$k(x,y)=\sum_{j=1}^\infty \sigma_j e_j(x)e_j(y)$
with respect to $\mu$, where $(e_j)_{j\ge 1}$ is  an orthonormal set in $L^2(\mu)$ and $\sigma_1\ge \sigma_2\ge \cdots \ge 0$.
For $d\ge 1$, define the following functions, which we use in the later technical statements and proofs:
\begin{itemize}
    \item $d$-dimensional feature map
        $\bm{\Phi}_d(x):=(\sqrt{\sigma_1}e_1(x),\dots,\sqrt{\sigma_d}e_d(x))^\top\in\R^d$,
    \item the residual kernel $k_{>d}(x,y):=\sum_{j>d}\sigma_j e_j(x)e_j(y)$ and its ``norm'' $g_d(x):=\sqrt{k_{>d}(x,x)}$,
    \item and the augmented feature map
        $\bm{\Psi}_d(x):=(\bm{\Phi}_d(x)^\top,g_d(x))^\top\in\R^{d+1}$.
\end{itemize}

The additional coordinate $g_d$ controls the residual kernel and allows us to transfer augmented
mean approximation to full worst-case error.

\begin{prop}[Transfer from augmented mean approximation to KQ]
\label{prop:augmented-transfer}
Let $\bm{w}=(w_i)_{i=1}^N\in\Delta_N$ and define $Q_N(f):=\sum_{i=1}^N w_i f(x_i)$. If
$\bigl\lvert\mathbb{E}_{x\sim\mu}[\bm{\Psi}_d(x)]-\sum_{i=1}^N w_i \bm{\Psi}_d(x_i)\bigr\rvert\le \varepsilon$
for some $\varepsilon\ge 0$, then
\[
\wce(Q_N;\Hil, \mu)\le 2\int g_d(x)\,\dd\mu(x)+\ve
\le 2\sqrt{\sum_{j>d}\sigma_j}+\ve.
\]
\end{prop}
\begin{proof}[Proof sketch]
It follows by carefully treating the truncation of the ``infinite-dimensional feature'' $(\sqrt{\sigma_j}e_j(x))_{j\ge1}$,
extending an existing argument treating the case $\sum_{i=1}^N w_i \bm{\Phi}_d(x_i)=\mathbb{E}_{x\sim\mu}[\bm{\Phi}_d(x)]$ \citep[Theorem~16]{hayakawa21b}.
Indeed, the first $d$ terms of the ``feature'' are precisely represented by $\bm\Phi_d(x)$ and the remaining term can be controlled by $g_d(x)$.
The full proof is given in Appendix~\ref{proof:augmented-transfer}.
\end{proof}

We can now state our main existential result for positive-weight KQ.

\begin{thm}[Analysis of positive-weight KQ via augmented convex-hull approximation]
\label{thm:kq-main}
Let positive integers $N, d$ satisfy $6(d+1)\le N$ and $x_1,\dots,x_N\sim\mu$ be independent.
Then, with probability at least $1-2^{-(d+1)}$, there exists a quadrature rule
$Q_N(f)=\sum_{i=1}^N w_i f(x_i)$, supported on $\{x_1,\dots,x_N\}$ with $(w_i)_{i=1}^N\in\Delta_N$ such that, for $n_d:=\left\lfloor \frac{N}{6(d+1)}\right\rfloor$
and
$\kappa:=\sup_{x\in\X}\sqrt{k(x, x)}$,
\[
\wce(Q_N;\Hil, \mu)\le 2\sqrt{\sum_{j>d}\sigma_j}+\frac{4\kappa}{n_d}.
\]
\end{thm}
\begin{proof}
    We can obtain the result by combining Proposition~\ref{prop:augmented-transfer}
    and Theorem~\ref{thm:main-convex-hull} regarding the $(d+1)$-dimensional random vector
    $\bm{\Psi}_d(x)$ with $x\sim\mu$.
    The full proof is given in Appendix~\ref{proof:kq-main}. 
\end{proof}
We then state the spectral consequences in terms of direct assumptions on the eigenvalue decay $(\sigma_j)_{j\ge1}$ rather than on the tail sum in the following remark. Its proof is given in Appendix~\ref{proof:kq-order}.

\begin{rem}[Spectral regimes and failure probability]
\label{rem:spectral-regimes}
Theorem~\ref{thm:kq-main} implies 
$
\wce(Q_N;\Hil, \mu)
=
\mathcal O\bigl(
(\sum_{j>d}\sigma_j)^{1/2}
+
d/N
\bigr)
$
with probability at least $1 - 2^{-(d+1)}$.
We can obtain bounds solely in terms of $N$ by balancing the two terms in the estimate;
such optimizations yield the following rates under spectral decay assumptions.
\begin{itemize}
    \item If
    $
    \sigma_j=\mathcal O(j^{-1-\beta})$ for some $\beta>0$,
    then
    $
    \wce(Q_N;\Hil, \mu)=\mathcal O\!\left(N^{-\beta/(\beta+2)}\right).
    $
    \item If
    $
    \sigma_j=\mathcal O(e^{-cj^{\gamma}})$ for some $c,\gamma>0$,
    then
    $
    \wce(Q_N;\Hil, \mu)=\ord{N^{-1}(\log N)^{1/\gamma}}.
    $
\end{itemize}
In the polynomial case, improvement over the Monte Carlo rate can be shown when $\beta>2$.

Although the failure probability $2^{-(d+1)}$ is not explicit in the displayed rates above,
it affects the rate only logarithmically.
Indeed, a size-$N$ pool contains $B$ independent blocks of size
$N_0=\lfloor N/B\rfloor$. Thus, applying Theorem~\ref{thm:kq-main} to each block and
taking the best block reduces the failure probability to at most $2^{-B(d+1)}$, while $N$ in the finite-sample term is replaced by
$N_0$. Hence, choosing
$B=\lceil \log_2(1/\delta)/(d+1)\rceil$ for $\delta<1$ and assuming
$N_0\ge 6(d+1)$ yields a finite-sample term of order
\[
    \frac{d}{N_0}
    =
    \ord{\frac{dB}{N}}
    =
    \ord{\frac{d(1 + \log_2(1/\delta)/(d+1))}{N}}
    =
    \ord{\frac{d+\log(1/\delta)}{N}}.
\]
Thus the failure-probability parameter $\delta$ affects the rate only logarithmically.
\end{rem}

\subsection{A constructive version via Frank--Wolfe}\label{sec:kq-construction}

Theorem~\ref{thm:kq-main} is existential: it shows that, with high probability, a good positive-weight rule supported on the evaluated pool exists. On a fixed pool, the optimal weight is the solution of the simplex-constrained convex quadratic program (CQP):
\begin{equation}
    \min_{\bm w\in\Delta_N}
    \left\|
    m_\mu-\sum_{i=1}^N w_i k(x_i,\cdot)
    \right\|_\Hil^2
    =
    \min_{\bm w\in\Delta_N}
    \bigl(
    \bm w^\top K \bm w - 2 \bm z^\top \bm w
    + \|m_\mu\|_{\mathcal H}^2
    \bigr),
    \label{eq:cqp}
\end{equation}
where $K=(k(x_i,x_j))_{i,j=1}^N$ and $\bm z=(m_\mu(x_i))_{i=1}^N$,
and the constant $\|m_\mu\|_{\mathcal H}^2$ can be dropped for finding weights.
Thus, Theorem~\ref{thm:kq-main} may be read as a statistical guarantee for this CQP objective.

To obtain a practical counterpart, we approximate this pool-supported positive optimum by the Frank--Wolfe algorithm~\citep{jaggi13}.
The application of Frank--Wolfe to KQ is not new: kernel herding \citep{che10} can be reinterpreted as a conditional-gradient method for a quadratic moment-discrepancy objective \citep{bac12}, and \citet{bri15} even proposes a BQ based on nodes selected by Frank--Wolfe.
However, the Frank--Wolfe in our context is used as a practical and lightweight way of approximately solving \eqref{eq:cqp}.

Let $\mathcal A_N:=\{k(x_1,\cdot),\dots,k(x_N,\cdot)\}\subset\Hil$ with $\mathcal M_N:=\cv(\mathcal A_N)$, and let $J_N(h):=\frac12\|m_\mu-h\|_\Hil^2$ for $h\in\mathcal M_N$. Minimizing $J_N$ over $\mathcal M_N$ is equivalent to minimizing the squared worst-case error over all KQs supported on the pool with weights in $\Delta_N$. We therefore run Frank--Wolfe on $J_N$ over $\mathcal M_N$ (concretely, Algorithm~\ref{algo:fw-kq}). The proof of the following theorem is given in Appendices~\ref{app:fw-rkhs}~\&~\ref{proof:kq-fw}.

\begin{thm}[Constructive positive-weight KQ]
\label{thm:kq-fw}
Assume the setting of Theorem~\ref{thm:kq-main}, and suppose that the target kernel mean can be evaluated on the pool, i.e.\ $m_\mu(x_i)=\int k(x,x_i)\,\dd\mu(x)$ is available for $i=1,\dots,N$. Let $\bm{w}^{(T)}$ be the output of Algorithm~\ref{algo:fw-kq} after $T$ steps, and define $Q_{N,T}(f):=\sum_{i=1}^N w_i^{(T)} f(x_i)$. Then, with probability at least $1-2^{-(d+1)}$, we have
\[
\wce(Q_{N,T};\Hil,\mu)
\le
2\sqrt{\sum_{j>d}\sigma_j}
+
\frac{4\kappa}{n_d}
+
\frac{4\kappa}{\sqrt{T+2}}.
\]
\end{thm}

% Thus $\wce(Q_{N,T};\Hil,\mu)=\mathcal O\!\left(\sqrt{\sum_{j>d}\sigma_j}+\frac{d}{N}+\frac{1}{\sqrt{T}}\right)$.
In particular, choosing $T\asymp (N/d)^2$ makes the additional optimization term $\ord{d/N}$, so the constructive bound matches the existential rate from Theorem~\ref{thm:kq-main} up to constants. Unlike the analysis in Section~\ref{sec:kq-existence}, the algorithm in Theorem~\ref{thm:kq-fw} \emph{does not} require access to the Mercer expansion or the truncation dimension $d$; these are used only to prove the existential guarantee.

\paragraph{Computational complexity.}
The CQP \eqref{eq:cqp} computes the exact positive optimum on the pool, but as a dense convex quadratic program it is not necessarily the most favorable route from a worst-case computational viewpoint. By contrast, if the Gram matrix $K$ and the vector $\bm z=(m_\mu(x_i))_{i=1}^N$ are precomputed, then each Frank--Wolfe iteration requires one linear minimization over the $N$ atoms together with an $\ord{N}$ update of the current scores. Thus the optimization cost is $\ord{NT}$ after an $\mathcal{O}(N^2)$ preprocessing step. With the choice $T\asymp (N/d)^2$, this becomes $\mathcal{O}(N^2+N^3/d^2)$, while retaining the same order in $N$ as Theorem~\ref{thm:kq-main}.

% \paragraph{Beyond the exact-target setting.}
% The exact-target assumption is convenient for the main statements, but the same framework extends directly to approximate kernel means. In particular, if $m_\mu$ is replaced by an approximation $\widehat m\in\Hil$, then Proposition~\ref{prop:aux-transfer} in Appendix~\ref{app:aux-transfer} shows that the total quadrature error is bounded by the sum of the approximation error $\|m_\mu-\widehat m\|_\Hil$ and the positive-reweighting error from the main construction.

\paragraph{Beyond the exact-target setting.}
The exact-target assumption is convenient for
the main statements, but the same fixed-pool viewpoint also yields a simple
robustness guarantee when only an approximate target kernel mean is available.
Suppose that the candidate pool is still sampled from $\mu$, but that the
algorithm uses a proxy kernel mean $m_{\hat\mu}$ given by the kernel mean of a large empirical proxy measure
$\hat\mu$. Let
$
    E_N(d)
    :=
    2(\sum_{j>d}\sigma_j)^{1/2}
    + 4\kappa / n_d
$
denote the bound in Theorem~\ref{thm:kq-main}. On the event of Theorem~6, there exists a positive KQ rule $Q^*_N$ such that
$\wce(Q_N^*;\Hil, \mu)\le E_N(d)$.
Therefore, from the triangle inequality, the same weights satisfy
$
    \wce(Q_N^*;\Hil,\hat{\mu})
    \le E_N(d)+\|m_\mu-m_{\hat\mu}\|_{\mathcal H}.
$
Consequently, if $Q_N$ is a KQ with weights minimizing the proxy fixed-pool objective given by replacing $\mu$ with $\hat{\mu}$ in \eqref{eq:cqp}, then we have
\begin{align}
    \wce(Q_N; \Hil, \mu)
    &\le \wce(Q_N; \Hil, {\hat\mu})
        + \|m_\mu-m_{\hat\mu}\|_{\mathcal H}\nonumber\\
    &\le 
    \wce(Q_N^*;\Hil,\hat{\mu})
        + \|m_\mu-m_{\hat\mu}\|_{\mathcal H}
    \le E_N(d)+2\|m_\mu-m_{\hat\mu}\|_{\mathcal H},
    \label{eq:triangle-inequality}
\end{align}
where the first inequality is by the triangle inequality, and the second inequality comes from the optimality of $Q_N$. 
Thus, approximate target kernel means introduce an additive error controlled by
the RKHS distance between the true and proxy kernel means. The same conclusion
holds for approximate solvers such as Frank--Wolfe, with an additional term of
optimization error.

\begin{algorithm}[t]
\caption{Positive-weight KQ via Frank--Wolfe on the RKHS atomic hull}
\label{algo:fw-kq}
\begin{algorithmic}[1]
\Require Kernel $k$, kernel mean values $m_\mu(x_1),\dots,m_\mu(x_N)$, pool points $x_1,\dots,x_N$, iterations $T$
\Ensure Positive weights $\bm w^{(T)}\in\Delta_N$
\State Initialize $i_0\in\arg\min_{i\in[N]} \bigl(k(x_i,x_i)-2m_\mu(x_i)\bigr)$
\State Set $\bm w^{(0)}:=\bm{1}_{i_0}$, where $\bm{1}_i\in\Delta_N$ denotes the $i$-th unit vector
\For{$t=0,1,\dots,T-1$}
    \State Compute $h^{(t)}(\cdot)=\sum_{i=1}^N w_i^{(t)}k(x_i,\cdot)$
    \State Select
    $
        i_t\in\arg\min_{i\in[N]}\bigl(h^{(t)}(x_i)-m_\mu(x_i)\bigr)
    $
    \State Set $\bm{s}^{(t)}:=\bm{1}_{i_t}$ and $\gamma_t:=2/(t+2)$
    \State Update
    $
        \bm w^{(t+1)}=(1-\gamma_t)\bm w^{(t)}+\gamma_t \bm{s}^{(t)}
    $
\EndFor
\State \Return $\bm w^{(T)}$
\end{algorithmic}
\end{algorithm}

\section{Experiments}\label{sec:experiments}

We use experiments to test whether the fixed-pool picture developed above is visible in practice.
For each point budget $N$, we construct a quadrature rule over multiple independent trials and report the RKHS worst-case error
$
    \wce(Q_N;\Hil,\mu)
    =
    \lVert m_\mu-\sum_{i=1}^N w_i k(x_i,\cdot)\rVert_{\Hil}
$
as a function of $N$.
We consider two benchmark families. The first consists of periodic Sobolev
kernels on $[0,1]^p$ with uniform target distribution, using
$(p,s)=(1,1),(1,3),(2,5)$, where $p$ is the input dimension and $s$ is the
smoothness. The second is an empirical RBF target in $\mathbb{R}^2$, given
by the uniform empirical measure on $10{,}000$ synthetic support points from
a four-component Gaussian mixture. In these settings, the target kernel mean
is available exactly on the pool, either analytically or by finite summation.
All main curves are averaged over $20$ independent trials, with
$N\in\{4,8,16,32,64,128\}$ for Sobolev and
$N\in\{4,8,16,32,64,128,256\}$ for the empirical RBF target.
We compare the following procedures:
\begin{itemize}
    \item \textbf{Fixed-pool CQP}: a numerical solution of the simplex-constrained CQP~\eqref{eq:cqp}, used as an approximation to the pool-supported positive optimum;
    \item \textbf{Fixed-pool FW}: our constructive method, which runs Frank--Wolfe directly in the RKHS over the convex hull of the pool atoms, with $T=N^2$ steps and the fixed step-size rule $\gamma_t=2/(t+2)$ unless otherwise stated;
    \item \textbf{Herding baselines}: depending on the experiment, either resample herding or global herding;
    \item \textbf{Monte Carlo}: equal-weight averaging on $N$ i.i.d.\ samples.
\end{itemize}
The comparison is meant to separate three effects: the statistical benefit of positive reweighting on a fixed evaluated pool, the optimization gap between Frank--Wolfe and the CQP on that pool, and the difference between our fixed-pool viewpoint and more classical herding-style constructions.
For the herding baselines, we use standard equal-weight kernel herding, with
fresh i.i.d. candidate sets for continuous Sobolev targets and the full support
for empirical targets. These baselines are not like-for-like fixed-pool methods,
but serve as standard constructive KQ references.

\paragraph{Why include herding baselines?}
Our theory is tailored to fixed-pool positive reweighting: an evaluated pool is given, and the task is to construct a stable quadrature rule supported on that pool.
Herding-type methods address a different problem, since they typically build the rule by sequential point selection rather than by reweighting a fixed pool.
We nevertheless include them as reference points because they are standard constructive baselines in kernel quadrature.
Thus, the comparison is not meant as a like-for-like match of assumptions, but rather as evidence that the fixed-pool positive approach remains competitive even relative to classical point-selection procedures.

In the main text, we include both smooth regimes, where the spectral advantage predicted by Section~\ref{sec:kq} should be most visible, and a lower-smoothness Sobolev regime as a challenging control.
Additional experimental details and further ablations, including timing, Frank--Wolfe step-size and budget ablations, and resample-herding candidate-size ablations, are deferred to Appendix~\ref{app:exp-details}.

\begin{figure}[t]
    \centering
    \includegraphics[width=\textwidth]{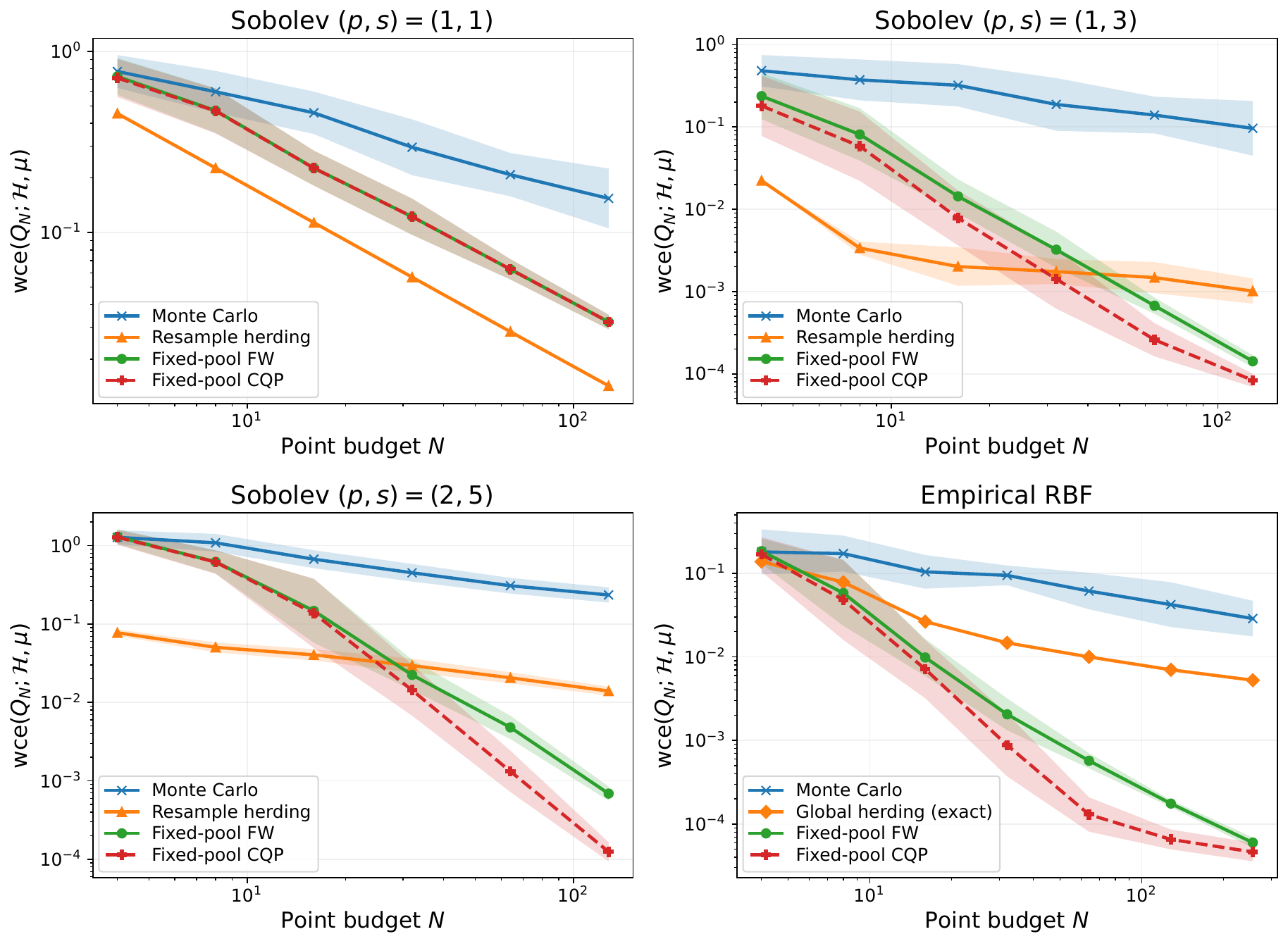}
    \caption{
    Main experimental comparison in terms of RKHS worst-case error.
    The first three panels show periodic Sobolev experiments with continuous target distribution, including one lower-smoothness control $(p,s)=(1,1)$ and two smoother regimes $(p,s)=(1,3)$ and $(p,s)=(2,5)$.
    The final panel shows an empirical-target experiment with an RBF kernel.
    We compare fixed-pool Frank--Wolfe with $T(N)=N^2$, the fixed-pool CQP, the appropriate herding baseline, and Monte Carlo.
    Shaded regions indicate standard deviation across trials, computed in log scale.
    }
    \label{fig:main-wce}
\end{figure}

Figure~\ref{fig:main-wce} shows the main results.
Across all settings, fixed-pool CQP and fixed-pool FW substantially improve over Monte Carlo, confirming that positive reweighting of an evaluated i.i.d.\ pool can exploit kernel structure beyond equal-weight averaging.
The fixed-pool FW curves closely track the fixed-pool CQP in the smoother Sobolev regimes and in the empirical RBF experiment, supporting the constructive role of Frank--Wolfe as a practical approximation to the pool-supported positive optimum.
In the lower-smoothness Sobolev case $(p,s)=(1,1)$, the gap between positive fixed-pool reweighting and Monte Carlo is smaller, and resample herding outperforms them.
This is consistent with the theory: the strongest spectral improvements are expected in favorable spectral regimes, while lower-smoothness cases provide a more challenging control.
In the smoother Sobolev regimes and in the empirical RBF experiment, the fixed-pool positive methods are markedly stronger at moderate and large budgets.
For the empirical target, global herding is competitive at small budgets, but fixed-pool FW and CQP achieve lower worst-case error overall.

\section{Limitations}\label{sec:limitations}
Our theoretical guarantees are tailored to fixed-pool kernel quadrature, where an evaluated candidate pool is already available and the goal is to construct a stable positive-weight rule supported on that pool. Accordingly, the paper does not provide a new theory of subsampling or sequential point selection in the more general kernel-quadrature setting. In that sense, our Frank--Wolfe procedure should be interpreted as a constructive approximation to the pool-supported positive optimum, rather than as a novel point-selection algorithm. Moreover, the guarantees themselves could further be sharpened in future work, since the empirical results suggest even faster convergence.

A second limitation is the exact-target assumption: evaluating $m_\mu$ exactly on the pool may itself be nontrivial. While an error estimate based on the triangle inequality~\eqref{eq:triangle-inequality} can be a point of compromise, this problem is universal for KQ methods, except for empirical-target or otherwise analytically tractable settings including Stein kernels~\citep{oat17,south2022semi,anastasiou2023stein}.

That said, the empirical comparisons with herding suggest that the fixed-pool positive approach remains competitive even relative to standard sequential baselines. Thus, whenever an evaluated pool is naturally available the method appears practically relevant beyond the exact scope of the theory.

\section{Conclusion}
We studied positive-weight kernel quadrature through random convex-hull
approximation of mean vectors. Our main result shows that, by using a
residual-diagonal augmentation of the truncated Mercer representation, one
obtains full RKHS error bounds of the form
$
    \wce(Q_N;\Hil,\mu)
    =
    \mathcal{O}\bigl((\sum_{j>d}\sigma_j)^{1/2} + d/N\bigr),
$
and hence rates beyond Monte Carlo in favorable spectral regimes despite
the simplex-weight constraint. We also gave a constructive counterpart by
running Frank--Wolfe directly in the RKHS on the convex hull of the pool
atoms, maintaining positive simplex weights together with explicit
optimization-error bounds. Overall, the paper gives both a geometric
explanation and a practical constructive route for stable positive
quadrature on an evaluated pool.

\section*{Acknowledgments}
The author is grateful to Harald Oberhauser and Terry Lyons for useful discussions on the topic.

\bibliography{cite}
\bibliographystyle{abbrvnat}
\appendix

\newpage

\section{A normalization viewpoint for unbounded diagonals}
\label{sec:bounded-reduction}

Suppose $k(x,x)>0$ holds for all $x\in\X$ and
$\int k(x,x)\,\dd\mu(x)<\infty$.
Write
\[
    r(x):=\sqrt{k(x,x)}
\]
and define the normalized kernel
\[
    \tilde k(x,y):=\frac{k(x,y)}{r(x)r(y)}.
\]
Then $\tilde k(x,x)\equiv 1$, so $\tilde k$ has bounded diagonal.
Let $\Hil$ and $\tilde\Hil$ denote the RKHSs associated with $k$ and $\tilde k$, respectively.
We also write $
    m_\rho^h := \int h(x,\cdot)\,\dd\rho(x)
$
for the kernel mean embedding of a probability measure $\rho$ with respect to a kernel $h$.
Thus $m_\mu^k$ is the original kernel mean embedding, while $m_{\tilde\mu}^{\tilde k}$ denotes the kernel mean embedding associated with the normalized kernel $\tilde k$ and the tilted measure $\tilde\mu$ introduced below.

The normalization induces an isometric correspondence between the two RKHSs.
Define
\[
    U:\tilde\Hil\to\Hil,
    \qquad
    (Ug)(x):=r(x)g(x).
\]
To see that $U$ is an isometry, first consider finite linear combinations
$g=\sum_{i=1}^m a_i\tilde k(x_i,\cdot)$.
Hence
\[
    \|Ug\|_{\Hil}^2
    =
    \nor{\sum_{i=1}^m a_i\frac{k(x_i,\cdot)}{r(x_i)}}_{\Hil}^2
    =
    \sum_{i,j=1}^m
    a_i a_j
    \frac{k(x_i,x_j)}{r(x_i)r(x_j)}
    =
    \sum_{i,j=1}^m
    a_i a_j
    \tilde k(x_i,x_j)
    =
    \|g\|_{\tilde\Hil}^2.
\]
Since the linear span of $\{\tilde{k}(x,\cdot)\mid x\in\X\}$ is dense in $\tilde\Hil$, the map $U$ extends uniquely to an isometry from $\tilde\Hil$ into $\Hil$.
Moreover, because $r(x)>0$, the linear span of $\{k(x,\cdot)/r(x)\mid x\in\X\}$ is the same as the linear span of $\{k(x,\cdot)\mid x\in\X\}$, and hence the range of $U$ is dense in $\Hil$.
Therefore $U$ identifies $\tilde\Hil$ and $\Hil$ isometrically.
In particular,
\[
    U^{-1}k(x,\cdot)=
    \frac1{r(\cdot)}k(x,\cdot)=
    r(x)\tilde k(x,\cdot).
\]

Now define the tilted probability measure
\[
    \dd\tilde\mu(x):=\frac{r(x)}{Z}\,\dd\mu(x),
    \qquad
    Z:=\int r(x)\,\dd\mu(x).
\]
The constant $Z$ is finite by the Cauchy--Schwarz inequality and the assumption
$\int k(x,x)\,\dd\mu(x)<\infty$.
Using the isometry above, the original kernel mean embedding satisfies
\[
    U^{-1}m_\mu^k
    =
    \int U^{-1}k(x,\cdot)\,\dd\mu(x)
    =
    \int r(x)\tilde k(x,\cdot)\,\dd\mu(x)
    =
    Z\int \tilde k(x,\cdot)\,\dd\tilde\mu(x)
    =
    Zm_{\tilde\mu}^{\tilde k}.
\]
Thus, after normalization, the target is the ordinary kernel mean embedding
$m_{\tilde\mu}^{\tilde k}$ for the bounded kernel $\tilde k$ and the tilted measure $\tilde\mu$, up to the scalar factor $Z$.
This gives the following transfer of KQ errors.
Let $x_1,\ldots,x_N\sim\tilde\mu$ and let $\bm\beta=(\beta_i)_{i=1}^N\in\Delta_N$.
Then, we have
\[
    \left\|
        m_\mu^k
        -
        \sum_{i=1}^N
        \frac{Z\beta_i}{r(x_i)} k(x_i,\cdot)
    \right\|_{\Hil}
    =
    Z
    \left\|
        m_{\tilde\mu}^{\tilde k}
        -
        \sum_{i=1}^N \beta_i \tilde k(x_i,\cdot)
    \right\|_{\tilde\Hil}.
\]
Consequently, a simplex-constrained positive quadrature rule $\tilde{Q}_N(g)=\sum_{i=1}^N\beta_ig(x_i)$
for the normalized problem
$(\tilde k,\tilde\mu)$ transfers to a positive weighted quadrature rule for the original problem
$(k,\mu)$, given by
\[
    Q_N(f)=\sum_{i=1}^N w_if(x_i),
    \qquad
    w_i=\frac{Z\beta_i}{r(x_i)}.
\]
These weights are positive, but they need not sum to one.
Thus, the normalization does not directly preserve the simplex constraint;
instead, it turns the original unbounded-diagonal problem into a weighted
positive KQ problem.

This change-of-measure viewpoint is related in spirit to weighted sampling
and importance-weighting ideas in KQ and BQ, including leverage-score
sampling~\citep{bac17,chatalic25}, weighted sampling
for BQ batch selection~\citep{ada22}, and
$\Pi$-importance sampling for Stein kernel quadrature~\citep{wang23steinpi}.
The distinction is that the present weighting arises from diagonal
normalization of an unbounded kernel and is used only as a reduction
viewpoint. A full theory for this weighted positive KQ setting is beyond
the scope of the present paper.

This viewpoint suggests a possible extension beyond the bounded-kernel
assumption, provided that sampling from $\tilde\mu$ is available and the
spectral decay of the normalized pair $(\tilde k,\tilde\mu)$ can be
controlled.

\section{Proofs}\label{app:proofs}

\subsection{Proof of Theorem~\ref{thm:one-sided}}\label{app:one-sided}
\begin{proof}%[Proog of Theorem~\ref{thm:one-sided}]
For convenience, let
\[
\sigma:=\sqrt{\E{W^2}}.
\]
We distinguish two cases based on the magnitude of $\sigma$.

\paragraph{Case 1: $\sigma \le \sqrt{2}M/\sqrt{n}$.}
Since $W_1,\dots,W_n$ are independent and centered,
\[
\E{\left(\frac1n\sum_{i=1}^n W_i\right)^2}
=
\frac{\sigma^2}{n}
\le
\frac{2M^2}{n^2}.
\]
Therefore, by Chebyshev's inequality,
\[
\p{\left|\frac1n\sum_{i=1}^n W_i\right|>\frac{2M}{n}}
\le
\frac{\E{\left(\frac1n\sum_{i=1}^n W_i\right)^2}}{(2M/n)^2}
\le \frac12.
\]
Hence
\[
\p{\frac1n\sum_{i=1}^n W_i \le \frac{2M}{n}} \ge \frac12.
\]

\paragraph{Case 2: $\sigma \ge \sqrt{2}M/\sqrt{n}$.}
Let $Y:=W/\sigma$ and $Y_i:=W_i/\sigma$. Then $\E{Y}=0$, $\E{Y^2}=1$, and
\[
\E{|Y|^3}
=
\frac{\E{|W|^3}}{\sigma^3}
\le
\frac{M\E{W^2}}{\sigma^3}
=
\frac{M}{\sigma}.
\]
Applying the Berry--Esseen inequality with a constant $1/2$~\citep[Corollary~1]{kor12}, we obtain
\[
\p{\frac1n\sum_{i=1}^n W_i \le \frac{2M}{n}}
=
\p{\frac1{\sqrt{n}}\sum_{i=1}^n Y_i \le \frac{2M}{\sigma\sqrt{n}}}
\ge
\p{Z\le \frac{2M}{\sigma\sqrt{n}}}
-
\frac{M}{2\sigma\sqrt{n}},
\]
where $Z$ is a standard Gaussian. Set
$
x:=\frac{2M}{\sigma\sqrt{n}}.
$
Under the current assumption, $x\in[0,\sqrt{2}]$. Thus it is enough to show that
\[
\p{Z\le x}-\frac{x}{4}\ge \frac12
\qquad\text{for all }x\in[0,\sqrt{2}].
\]
Define $f(x):=\p{Z\le x}-\frac{x}{4}$,
then we have
\[
f'(x)=\frac{1}{\sqrt{2\pi}}e^{-x^2/2}-\frac14,
\qquad
f''(x)=-\frac{x}{\sqrt{2\pi}}e^{-x^2/2},
\]
so $f$ is concave on $[0,\sqrt{2}]$. Therefore it suffices to check the endpoints. We have $f(0)=1/2$. Moreover,
\[
f(\sqrt{2})-\frac12
=
f(\sqrt{2})-f(0)
=
\int_0^{\sqrt{2}}f^\prime(u)\,\dd u
=
\frac{1}{\sqrt{2\pi}}\int_0^{\sqrt{2}}e^{-u^2/2}\,\dd u-\frac{1}{\sqrt{8}}.
\]
Using $e^{-t}\ge 1-t$ for $t\in\R$, we have
\[
\frac{1}{\sqrt{2\pi}}\int_0^{\sqrt{2}}e^{-u^2/2}\,\dd u
\ge
\frac{1}{\sqrt{2\pi}}\int_0^{\sqrt{2}}\left(1-\frac{u^2}{2}\right)\dd u
=
\frac1{\sqrt{2\pi}}\left(\sqrt2 - \frac{2\sqrt{2}}6\right)
=
\frac{2}{3\sqrt{\pi}}.
\]
Hence
\[
f(\sqrt{2})-\frac12
\ge
\frac{2}{3\sqrt{\pi}}-\frac{1}{\sqrt{8}}
=
\frac{1}{\sqrt{8\pi}}\left(\frac{\sqrt{32}}{3}-\sqrt{\pi}\right)>0,
\]
since $32/9=3.55\ldots>\pi$. Therefore $f(x)\ge 1/2$ on $[0,\sqrt{2}]$, and the theorem follows.
\end{proof}

\subsection{Proof of Proposition~\ref{prop:augmented-transfer}}\label{proof:augmented-transfer}

\begin{proof}
First, we approximate the kernel by the inner product of the truncated feature map
$\bm{\Phi}_d$ as follows:
\begin{align}
    \left\lvert k(x, y)
    - \bm{\Phi}_d(x)^\top\bm{\Phi}_d(y)\right\rvert
    &=\left\lvert
        \sum_{j>d}\sigma_je_j(x)e_j(y)
    \right\rvert \nonumber\\
    &\le \sqrt{\sum_{j>d}\sigma_je_j(x)^2}
    \sqrt{\sum_{j>d}\sigma_je_j(y)^2}
    =g_d(x)g_d(y),
    \label{eq:feature-remainder}
\end{align}
where the inequality follows from the Cauchy--Schwarz inequality for sequences.

By exploiting the norm form of the worst-case error \eqref{eq:wce-norm-form},
we have
\begin{align*}
    \wce(Q_N; \Hil,\mu)^2
    &= \ip{\int_\X k(x,\cdot)\,\dd\mu(x)
        - \sum_{i=1}^N w_i k(x_i, \cdot),
        \int_\X k(x,\cdot)\,\dd\mu(x)
        - \sum_{i=1}^N w_i k(x_i, \cdot)}_\Hil\\
    &= \int_\X\int_\X k(x, y)\,\dd\mu(x)\,\dd\mu(y)
        - 2\sum_{i=1}^Nw_i\int_\X k(x, x_i)\,\dd\mu(x)
        + \sum_{i,j=1}^Nw_iw_jk(x_i,x_j).
\end{align*}
If we replace $k(x,y)$ with
$\bm{\Phi}_d(x)^\top\bm{\Phi}_d(y)$ in the right-hand side,
the resulting expression is
\[
    \left\lvert
        \mathbb{E}_{x\sim\mu}[\bm{\Phi}_d(x)]
        - \sum_{i=1}^Nw_i\bm{\Phi}_d(x_i)
    \right\rvert^2 .
\]
Thus, from \eqref{eq:feature-remainder}, we obtain
\begin{align*}
    &\left\lvert\wce(Q_N; \Hil,\mu)^2
    - \left\lvert
        \mathbb{E}_{x\sim\mu}[\bm{\Phi}_d(x)]
        - \sum_{i=1}^Nw_i\bm{\Phi}_d(x_i)
    \right\rvert^2\right\rvert\\
    &\le \int_\X\int_\X g_d(x)g_d(y)\,\dd\mu(x)\,\dd\mu(y)
        + 2\sum_{i=1}^Nw_i\int_\X g_d(x)g_d(x_i)\,\dd\mu(x)
        + \sum_{i,j=1}^Nw_iw_jg_d(x_i)g_d(x_j)\\
    &=
    \left(
        \int_\X g_d(x)\,\dd\mu(x)
        + \sum_{i=1}^Nw_ig_d(x_i)
    \right)^2 \\
    &\le
    \left(
        2\int_\X g_d(x)\,\dd\mu(x)
        +
        \left\lvert
            \mathbb{E}_{x\sim\mu}[g_d(x)]
            - \sum_{i=1}^Nw_ig_d(x_i)
        \right\rvert
    \right)^2.
\end{align*}
Therefore,
\begin{align*}
    &\wce(Q_N;\Hil,\mu)\\
    &\le
    \sqrt{
        \left\lvert
            \mathbb{E}_{x\sim\mu}[\bm{\Phi}_d(x)]
            - \sum_{i=1}^Nw_i\bm{\Phi}_d(x_i)
        \right\rvert^2
        +
        \left(
            2\int_\X g_d(x)\,\dd\mu(x)
            +
            \left\lvert
                \mathbb{E}_{x\sim\mu}[g_d(x)]
                - \sum_{i=1}^Nw_ig_d(x_i)
            \right\rvert
        \right)^2
    } \\
    &\le
    2\int_\X g_d(x)\,\dd\mu(x)
    +
    \sqrt{
        \left\lvert
            \mathbb{E}_{x\sim\mu}[\bm{\Phi}_d(x)]
            - \sum_{i=1}^Nw_i\bm{\Phi}_d(x_i)
        \right\rvert^2
        +
        \left\lvert
            \mathbb{E}_{x\sim\mu}[g_d(x)]
            - \sum_{i=1}^Nw_ig_d(x_i)
        \right\rvert^2
    } \\
    &=
    2\int_\X g_d(x)\,\dd\mu(x)
    +
    \left\lvert
        \mathbb{E}_{x\sim\mu}[\bm{\Psi}_d(x)]
        - \sum_{i=1}^Nw_i\bm{\Psi}_d(x_i)
    \right\rvert \\
    &\le 2\int_\X g_d(x)\,\dd\mu(x) + \ve,
\end{align*}
where the second inequality follows from the elementary inequality
$\sqrt{u^2+(v+w)^2}\le v+\sqrt{u^2+w^2}$ for $u,v,w\ge0$.
Indeed,
\[
    (\mathrm{RHS})^2 - (\mathrm{LHS})^2
    =
    2v\sqrt{u^2+w^2} - 2vw
    =
    2v\bigl(\sqrt{u^2+w^2} - w\bigr)\ge 0.
\]
Finally, by the Cauchy--Schwarz inequality,
\[
    \int_\X g_d(x)\,\dd\mu(x)
    =
    \int_\X \sqrt{\sum_{j>d}\sigma_je_j(x)^2}\,\dd\mu(x)
    \le
    \sqrt{
        \int_\X \sum_{j>d}\sigma_je_j(x)^2\,\dd\mu(x)
    }
    =
    \sqrt{\sum_{j>d}\sigma_j}.
\]
Combining the last two displays gives the desired estimate.
\end{proof}

\subsection{Proof of Theorem~\ref{thm:kq-main}}\label{proof:kq-main}
\begin{proof}%[Proog of Theorem~\ref{thm:kq-main}]
Recall that $\bm{\Psi}_d(x)=(\bm{\Phi}_d(x),g_d(x))\in\R^{d+1}$
and $g_d(x)=\sqrt{\sum_{j>d}\sigma_je_j(x)^2}$. By construction, we have
\[
\lvert \bm{\Psi}_d(x)\rvert^2
=
\sum_{j=1}^d \sigma_j e_j(x)^2
+ \sum_{j>d} \sigma_j e_j(x)^2
=
k(x,x)
\le
\kappa^2,
\]
so $\lvert\bm{\Psi}_d(x)\rvert\le \kappa$ for all $x\in\mathcal X$.

Set $\bm Z:=\bm{\Psi}_d(x)-\mathbb{E}_{x\sim\mu}[\bm{\Psi}_d(x)]$.
Then $\E{\bm Z}=\bm 0$, and
$
\lvert\bm Z\rvert
\le
\lvert \bm{\Psi}_d(x)\rvert+
\mathbb{E}_{x\sim\mu}[\lvert\bm{\Psi}_d(x)\rvert]
\le
2\kappa
$.
Set $n_d:=\left\lfloor \frac{N}{6(d+1)}\right\rfloor$. Then $N\ge 6(d+1)n_d$.
By applying Theorem~\ref{thm:main-convex-hull} to the centered $(d+1)$-dimensional random vector $\bm Z$, we obtain that, with probability at least $1-2^{-(d+1)}$, there exists $\bm w\in\Delta_N$ such that
\[
\left\lvert
\sum_{i=1}^N w_i \bm Z_i
\right\rvert
\le
\frac{4\kappa}{n_d},
\]
where $\bm{Z}_i=\bm\Psi_d(x_i) -\mathbb{E}_{x\sim\mu}[\bm{\Psi}_d(x)]$ with $x_i\sim_{\text{iid}}\mu$.
Equivalently, we have
\[
\left\lvert
\mathbb{E}_{x\sim\mu}[\bm{\Psi}_d(x)]-\sum_{i=1}^N w_i \bm{\Psi}_d(x_i)
\right\rvert
\le
\frac{4\kappa}{n_d}.
\]
Thus, applying Proposition~\ref{prop:augmented-transfer} with $\varepsilon=4\kappa/n_d$ and $Q_N(f):=\sum_{i=1}^Nw_if(x_i)$ gives
\[
\wce(Q_N;\Hil, \mu)
\le
2\sqrt{\sum_{j>d}\sigma_j}+\frac{4\kappa}{n_d},
\]
which completes the proof.
\end{proof}

\subsection{Proof of Remark~\ref{rem:spectral-regimes}}\label{proof:kq-order}

\begin{proof}%[Proog of Remark~\ref{rem:spectral-regimes}]
By Theorem~\ref{thm:kq-main},
\[
\wce(Q_N;\Hil, \mu)
=
\mathcal O\!\left(
\sqrt{\sum_{j>d}\sigma_j}
+
\frac{d}{N}
\right).
\]

\noindent(1) Suppose first that
$
\sigma_j=\mathcal O(j^{-1-\beta})
$ for $\beta>0$.
Since $\sum_{j>d} j^{-1-\beta}
\le \int_d^\infty t^{-1-\beta}\,\dd t = d^{-\beta}/\beta$,
we have
$
    \sum_{j>d}\sigma_j
    =
    \mathcal O(d^{-\beta}).
$
Hence,
\[
\wce(Q_N;\Hil, \mu)
=
\mathcal O\!\left(d^{-\beta/2}+\frac{d}{N}\right).
\]
Balancing the two terms gives
$
d\asymp N^{2/(\beta+2)},
$
and therefore
\[
\wce(Q_N;\Hil, \mu)
=
\mathcal O\!\left(N^{-\beta/(\beta+2)}\right).
\]

\noindent(2) Next, suppose that
$
    \sigma_j=\mathcal O(e^{-cj^\gamma})
$
for some $c,\gamma>0$. Then, we have
\[
    \sum_{j>d}\sigma_j
    =
    \mathcal O\!\left(\sum_{j>d} e^{-cj^\gamma}\right)
    \le\mathcal O\!\left(\int_d^\infty e^{-ct^\gamma}\,\dd t\right).
\]
By the change of variables $u=ct^\gamma$, we obtain
\[
    \int_d^\infty e^{-ct^\gamma}\,\dd t
    =
    \frac{1}{\gamma\, c^{1/\gamma}}
    \int_{cd^\gamma}^\infty u^{1/\gamma-1}e^{-u}\,\dd u
    =
    \mathcal O\!\left(d^{1-\gamma}e^{-cd^\gamma}\right),
\]
and therefore
\[
    \sqrt{\sum_{j>d}\sigma_j}
    =
    \mathcal O\!\left(d^{(1-\gamma)/2}e^{-cd^\gamma/2}\right).
\]
Hence, Theorem~\ref{thm:kq-main} yields
\[
    \wce(Q_N;\Hil,\mu)
    =
    \mathcal O\!\left(
        d^{(1-\gamma)/2}e^{-cd^\gamma/2}
        +
        \frac{d}{N}
    \right).
\]
Now choose
\[
    d:=\left\lceil \left(\frac{2\log N}{c}\right)^{1/\gamma}\right\rceil.
\]
Then,
$
    d^\gamma \ge \frac{2\log N}{c},
$
so
$
    e^{-cd^\gamma/2}\le e^{-\log N}=1/N.
$
Therefore, we have
\[
    d^{(1-\gamma)/2}e^{-cd^\gamma/2}
    \le
    \frac{d^{(1-\gamma)/2}}{N}.
\]
We have
\[
    d
    \le
    \left(\frac{2\log N}{c}\right)^{1/\gamma}+1
    =
    \mathcal O\!\left((\log N)^{1/\gamma}\right).
\]
Since $d\ge 1$, we have
$d^{(1-\gamma)/2}\le d$,
and thus
\[
    d^{(1-\gamma)/2}e^{-cd^\gamma/2}
    =
    \mathcal O\!\left(\frac{(\log N)^{1/\gamma}}{N}\right).
\]
Moreover,
\[
    \frac{d}N
    =
    \mathcal O\!\left(\frac{(\log N)^{1/\gamma}}{N}\right).
\]
Combining the two bounds, we conclude that
\[
    \wce(Q_N;\Hil,\mu)
    =
    \mathcal O\!\left(\frac{(\log N)^{1/\gamma}}{N}\right).\qedhere
\]
\end{proof}

\subsection{Frank--Wolfe in the RKHS atomic hull}\label{app:fw-rkhs}

Algorithm~\ref{algo:fw-kq} shows the Frank--Wolfe procedure used in
Theorem~\ref{thm:kq-fw}. This is a standard conditional-gradient method
applied to the RKHS objective
\[
J_N(h):=\frac12\|m_\mu-h\|_\Hil^2
\]
over the convex hull
\[
\mathcal A_N:=\{k(x_1,\cdot),\dots,k(x_N,\cdot)\}\subset\Hil,
\qquad
\mathcal M_N:=\cv(\mathcal A_N).
\]

\begin{thm}[Frank--Wolfe for pool-supported KQ]
\label{thm:fw-kernel}
Under the above setting,
let $h^{(T)}:=\sum_{i=1}^N w_i^{(T)}k(x_i,\cdot)$ be the output of
Algorithm~\ref{algo:fw-kq}. If $\sup_{x\in\mathcal X}\sqrt{k(x,x)}\le \kappa$, then
for every $T\ge0$,
\[
J_N(h^{(T)})-\inf_{h\in\mathcal M_N}J_N(h)
\le
\frac{8\kappa^2}{T+2}.
\]
\end{thm}

\begin{proof}
While \citet{jaggi13} gives the general framework,
we give a short proof specialized to the present quadratic objective for completeness.

Let $h^*\in\arg\min_{h\in\mathcal M_N}J_N(h)$, and define
\[
\Delta_t:=J_N(h^{(t)})-J_N(h^*).
\]
Because of the reproducing property and $\mathcal M_N=\cv(\mathcal A_N)$, we have
\[
    \min_{i\in[N]}\{h^{(t)}(x_i)-m_\mu(x_i)\}
    = \min_{i\in[N]}\langle h^{(t)}-m_\mu,k(x_i,\cdot)\rangle_\Hil
\]
Thus, Algorithm~\ref{algo:fw-kq} updates $h^{(t)}$
as
\begin{equation}
    h^{(t+1)} = (1-\gamma_t)h^{(t)} + \gamma_t s^{(t)},
    \qquad s^{(t)}\in\mathop\mathrm{argmin}_{s\in\cv(\mathcal{A}_N)}\langle h^{(t)}-m_\mu, s\rangle_\Hil.
    \label{eq:fw-min-update}
\end{equation}

Now fix $t\ge0$. Since $J_N$ is quadratic, we have
\[
J_N\bigl(h^{(t+1)}\bigr)
=
J_N(h^{(t)})
+
\gamma_t\langle h^{(t)}-m_\mu,\,s^{(t)}-h^{(t)}\rangle_\Hil
+
\frac{\gamma_t^2}{2}\|s^{(t)}-h^{(t)}\|_\Hil^2.
\]
From the assumption, each atom satisfies
$\|k(x_i,\cdot)\|_\Hil=\sqrt{k(x_i,x_i)}\le \kappa$,
so every element of $\mathcal M_N$ has norm at most $\kappa$, and hence
$
\|s^{(t)}-h^{(t)}\|_\Hil\le 2\kappa.
$
Therefore, we obtain
\begin{equation}
    J_N\bigl(h^{(t+1)}\bigr)
    \le
    J_N(h^{(t)})
    +
    \gamma_t\langle h^{(t)}-m_\mu,\,s^{(t)}-h^{(t)}\rangle_\Hil
    +
    2\gamma_t^2\kappa^2.
    \label{eq:fw-1}
\end{equation}
Moreover, by optimality of $s^{(t)}$ in \eqref{eq:fw-min-update},
we have
\begin{equation}
    \langle h^{(t)}-m_\mu,\,s^{(t)}\rangle_\Hil
    \le
    \langle h^{(t)}-m_\mu,\,h^*\rangle_\Hil.
    \label{eq:fw-2}
\end{equation}
Hence, we have
\[
\langle h^{(t)}-m_\mu,\,s^{(t)}-h^{(t)}\rangle_\Hil
\le
\langle h^{(t)}-m_\mu,\,h^*-h^{(t)}\rangle_\Hil
\le
J_N(h^*)-J_N(h^{(t)})
=
-\Delta_t,
\]
where the second inequality can be derived from
the identity $\ip{a, b-a} = \frac12(\nor{b}^2-\nor{a}^2-\nor{b-a}^2)$ with
$a = h^{(t)}-m_\mu$ and $b=h^*-m_\mu$.
Therefore, by using the Frank--Wolfe update
$h^{(t+1)}=(1-\gamma_t)h^{(t)}+\gamma_t s^{(t)}$
with $\gamma_t=2/(t+2)$, we obtain
\begin{equation}
    \Delta_{t+1}
    \le
    (1-\gamma_t)\Delta_t+2\gamma_t^2\kappa^2
    =
    \frac{t}{t+2}\Delta_t+\frac{8\kappa^2}{(t+2)^2}.
    \label{eq:fw-recurr}
\end{equation}

We now prove by induction that
\begin{equation}
    \Delta_t\le \frac{8\kappa^2}{t+2}
    \qquad\text{for all }t\ge0.
    \label{eq:fw-ih}
\end{equation}
For $t=0$, since $\|h^{(0)}\|_\Hil\le \kappa$ and $\|m_\mu\|_\Hil\le \kappa$, we have
\[
\Delta_0\le J_N(h^{(0)})=\frac12\|m_\mu-h^{(0)}\|_\Hil^2\le 2\kappa^2
\le \frac{8\kappa^2}{2}.
\]
Assume the \eqref{eq:fw-ih} for $t$. Then, for $t+1$, \eqref{eq:fw-recurr} becomes
\[
    \Delta_{t+1}
    \le
    \frac{t}{t+2}\cdot \frac{8\kappa^2}{t+2}
    +\frac{8\kappa^2}{(t+2)^2}
    =
    \frac{8\kappa^2(t+1)}{(t+2)^2}
    \le
    \frac{8\kappa^2}{t+3},
\]
since $(t+1)(t+3)\le (t+2)^2$, which proves the induction.

Therefore, we obtain
\[
    J_N(h^{(T)})-\inf_{h\in\mathcal M_N}J_N(h)
    =
    \Delta_T
    \le
    \frac{8\kappa^2}{T+2},
\]
which is the desired conclusion.
\end{proof}

\subsection{Proof of Theorem~\ref{thm:kq-fw}}
\label{proof:kq-fw}

\begin{proof}%[Proog of Theorem~\ref{thm:kq-fw}]
On the event from Theorem~\ref{thm:kq-main}, which has probability at least
$1-2^{-(d+1)}$, the optimal positive-weight rule supported on
$\{x_1,\dots,x_N\}$ satisfies
\[
    \inf_{h\in\mathcal M_N}J_N(h)
    = \frac12 \inf_{\bm w\in\Delta_N}\wce(Q_{N,\bm w};\Hil,\mu)^2
    \le
    \frac12\left(2\sqrt{\sum_{j>d}\sigma_j}+\frac{4\kappa}{n_d}\right)^2.
\]
By combining Theorem~\ref{thm:fw-kernel} with it,
we have
\[
    J_N(h^{(T)})
    \le
    \inf_{h\in\mathcal M_N}J_N(h) + \frac{8\kappa^2}{T+2}
    \le
    \frac12\left(2\sqrt{\sum_{j>d}\sigma_j}+\frac{4\kappa}{n_d}\right)^2
    +\frac{8\kappa^2}{T+2}.
\]
Since $\wce(Q_{N,T};\Hil,\mu)^2 = 2J_N(h^{(T)})$,
we obtain
\[
    \wce(Q_{N,T};\Hil,\mu)^2
    \le
    \left(2\sqrt{\sum_{j>d}\sigma_j}+\frac{4\kappa}{n_d}\right)^2
    +\frac{16\kappa^2}{T+2}.
\]
Taking square roots and using $\sqrt{a^2+b^2}\le a+b$ for $a,b\ge0$, we have
\[
\wce(Q_{N,T};\Hil,\mu)
\le
2\sqrt{\sum_{j>d}\sigma_j}
+\frac{4\kappa}{n_d}
+\frac{4\kappa}{\sqrt{T+2}},
\]
which completes the proof.
\end{proof}

\section{Additional experimental details}
\label{app:exp-details}

This appendix provides the details needed to reproduce and interpret the experiments in Section~\ref{sec:experiments}.
The main-text results are reported in Figure~\ref{fig:main-wce}.
Here we provide the common protocol, benchmark definitions, implementation details, and three additional experimental views: wall-clock time, Frank--Wolfe budget and step-size ablations, and resample-herding candidate-size ablations.
A full rerun of all experiments reported in this paper would take roughly one hour on a single MacBook Air M2.

\subsection{Common protocol}

For each point budget $N$, we run $20$ independent trials unless otherwise stated and report the RKHS worst-case error
\[
    \wce(Q_N;\Hil,\mu)
    =
    \left\|m_\mu-\sum_{i=1}^N w_i k(x_i,\cdot)\right\|_{\Hil}.
\]
For the fixed-pool methods, the same sampled candidate pool is shared by the fixed-pool CQP and Frank--Wolfe within each trial.
Thus, their comparison reflects optimization quality on a common support set rather than variability due to different sampled pools.

In all plots, the shaded bands show standard deviation across trials.
Since the vertical axis is logarithmic, the bands are computed in log scale to avoid distorted lower envelopes near zero.
The elapsed-time plots use total wall-clock time, including both rule construction and final worst-case-error evaluation.

\subsection{Information required by the constructive methods}

For clarity, we distinguish the quantities used in the analysis from those required by the actual algorithms.
The analysis in Section~\ref{sec:kq} introduces the Mercer expansion of the kernel and a truncation dimension $d$.
These quantities are used only to derive the existence and convergence guarantees.
They are not required as algorithmic inputs.

In contrast, the fixed-pool CQP and the Frank--Wolfe construction require only:
\begin{itemize}
    \item an evaluated pool $x_1,\dots,x_N$;
    \item kernel evaluations $k(x_i,x_j)$ on that pool;
    \item the target kernel mean evaluated on the pool, $m_\mu(x_i)=\int k(x,x_i)\,\dd\mu(x)$.
\end{itemize}
In the empirical-target setting, these quantities can be computed exactly from the reference sample.
In the periodic Sobolev experiments, they are also available in closed form or through the high-accuracy numerical evaluation described below.

\subsection{Periodic Sobolev benchmark}

In the periodic Sobolev experiment, we consider $[0,1]^p$ with periodic boundary condition, the uniform target distribution, and a periodic Sobolev kernel of smoothness parameter $s$.
In one dimension, for $x,y\in[0,1]$, the kernel is
\[
    k_s(x, y)
    =
    1
    +
    \frac{(-1)^{s-1}(2\pi)^{2s}}{(2s)!}
    B_{2s}(\lvert x-y\rvert),
\]
where $B_{2s}$ is the Bernoulli polynomial~\citep{wah90}.
For dimension $p\ge2$, we use the coordinate-wise product kernel.
The target kernel mean is available exactly in this setting.

We compare fixed-pool CQP, fixed-pool Frank--Wolfe with $T(N)=N^2$, resample herding, and Monte Carlo over the point budgets
\[
    N\in\{4,8,16,32,64,128\}.
\]
For resample herding, the default candidate-set size is $R=4096$.
Instead of attempting global optimization over $[0,1]^p$, resample herding optimizes over a fresh set of $R$ i.i.d.\ points drawn from the target distribution at each iteration.

\paragraph{Numerical evaluation of higher-order periodic Sobolev kernels.}
For the periodic Sobolev experiments, low smoothness orders are evaluated in closed form.
For higher smoothness orders, namely $s>3$, we instead evaluate the kernel numerically via a truncated Fourier representation.
This is a numerical device for stable kernel evaluation; it is not an additional approximation step in the quadrature algorithm itself.

Concretely, in one dimension we use the periodic Sobolev factor
\[
    K_s(t)=1+2\sum_{m\ge 1}\frac{\cos(2\pi m t)}{m^{2s}},
    \qquad t\in[0,1],
\]
and truncate the series at a level chosen automatically so that the tail is below the prescribed tolerance $10^{-12}$.
In higher dimensions, the kernel is formed by the standard tensor-product construction across coordinates.

Within each run, the same kernel implementation is used consistently for Gram-matrix construction, kernel-mean evaluation, and worst-case-error evaluation.
Thus, the reported results correspond to a single numerically well-defined kernel throughout.
The tolerance $10^{-12}$ was chosen so that the resulting kernel-approximation error is negligible relative to the scale of the reported worst-case errors.
In the main-text periodic Sobolev experiments, the cases $(p,s)=(1,1)$ and $(p,s)=(1,3)$ are evaluated in closed form, while the higher-order case $(p,s)=(2,5)$ uses this Fourier-based implementation.

\subsection{Empirical RBF benchmark}

In the empirical RBF experiment, the target measure is an empirical distribution
\[
    \hat\mu_M=\frac1M\sum_{j=1}^M \delta_{z_j},
\]
where $\{z_j\}_{j=1}^M$ is a reference sample of size $M=10000$ sampled from a Gaussian mixture.
In the experiments reported here, we use a two-dimensional mixture with four components, radius parameter $2.5$, and component standard deviation $0.35$.
We use an RBF kernel
\[
    k(x,y)
    =
    \exp\!\left(-\frac{\|x-y\|^2}{2\ell^2}\right),
\]
with $\ell$ determined by the median heuristic~\citep{gar17}.
In this discrete-target setting, the target kernel mean can be evaluated exactly on candidate points, and global herding is available as a baseline.

We compare fixed-pool CQP, fixed-pool Frank--Wolfe with $T(N)=N^2$, global herding, and Monte Carlo over the point budgets
\[
    N\in\{4,8,16,32,64,128,256\}.
\]

\subsection{Implementation and timing details}

For Frank--Wolfe, the main experiments use the standard fixed step-size schedule $\gamma_t=2/(t+2)$.
The simplex-constrained CQP baseline was solved with
\texttt{scipy.optimize.minimize} using SLSQP. We used the uniform weight vector as initialization, imposed the bounds $w_i\ge 0$ and equality constraint
$\sum_i w_i=1$ for weights, and used $\texttt{ftol}=10^{-12}$ and $\texttt{maxiter}=5000$ as solver settings.
All experiments were run on a MacBook Air with an M2 chip.
The timing results should therefore be interpreted as implementation-dependent wall-clock measurements rather than optimized runtime benchmarks.
% The experimental scripts used to generate the reported figures are included in the supplementary material; the main and appendix experiments are generated by the corresponding shell scripts for the main experiments and appendix ablations.

\subsection{Wall-clock time and accuracy tradeoff}

\begin{figure}[h]
    \centering
    \includegraphics[width=\textwidth]{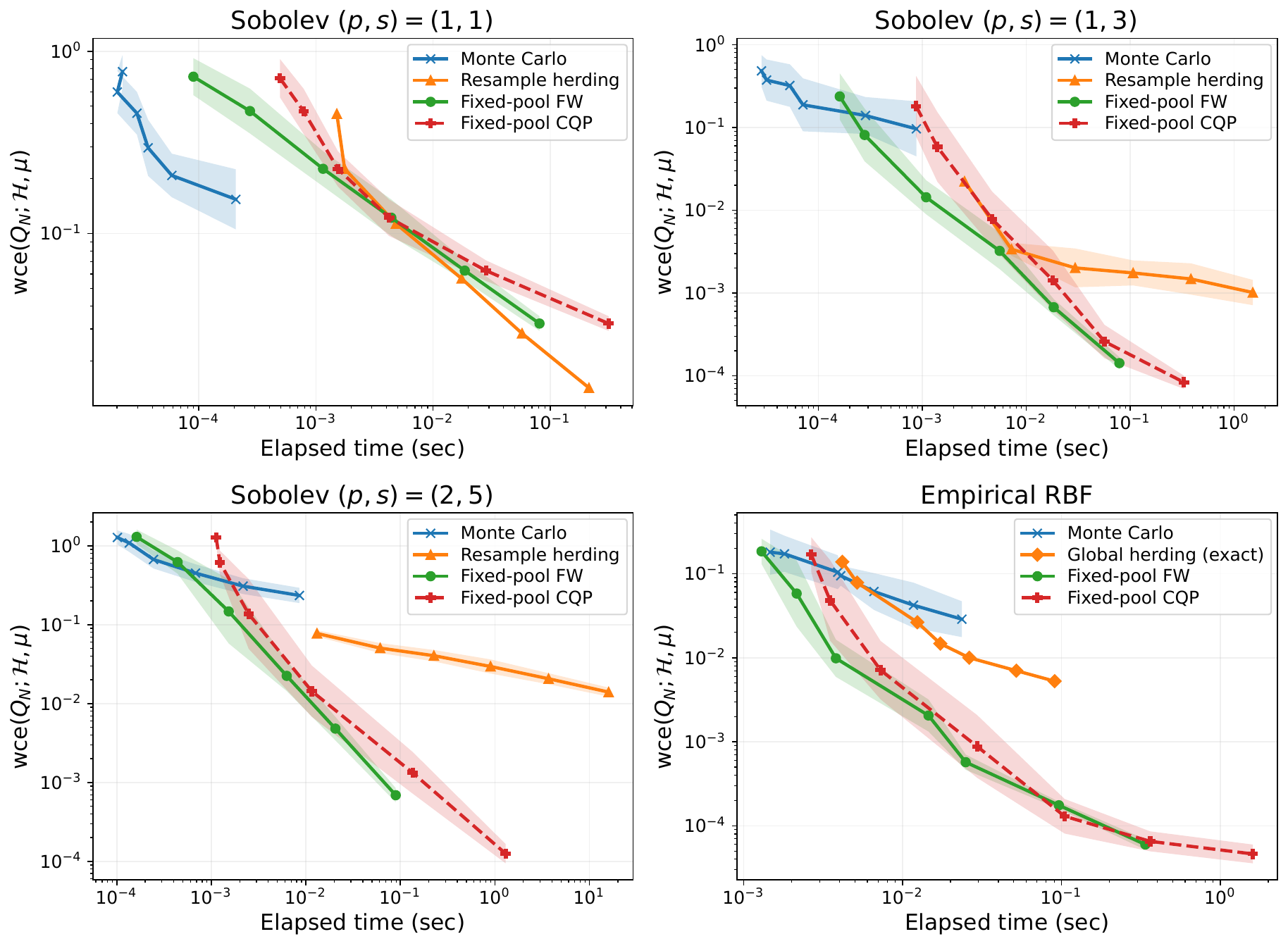}
    \caption{
    End-to-end wall-clock time and accuracy tradeoff for the main experimental settings.
    The horizontal axis shows total elapsed time, including both rule construction and final worst-case-error evaluation.
    The plotted methods are the same as in Figure~\ref{fig:main-wce}.
    }
    \label{fig:time-tradeoff}
\end{figure}

Figure~\ref{fig:time-tradeoff} reports the same main experimental settings as Figure~\ref{fig:main-wce}, but with elapsed time on the horizontal axis.
Monte Carlo is fastest but gives the largest errors.
The fixed-pool CQP is statistically strongest, but solving a dense quadratic program is relatively expensive.
Fixed-pool Frank--Wolfe provides a more favorable accuracy-time tradeoff: it closely approaches the CQP in the smoother Sobolev and empirical RBF settings while avoiding the full CQP solve.
The herding baselines can be competitive in some regimes, especially in the lower-smoothness Sobolev setting, but their practical cost depends on the candidate search problem.
In particular, resample herding becomes less favorable in the higher-dimensional Sobolev setting, where drawing and scanning large candidate sets is costly.

\subsection{Frank--Wolfe budget and step-size ablation}

\begin{figure}[h]
    \centering
    \includegraphics[width=\textwidth]{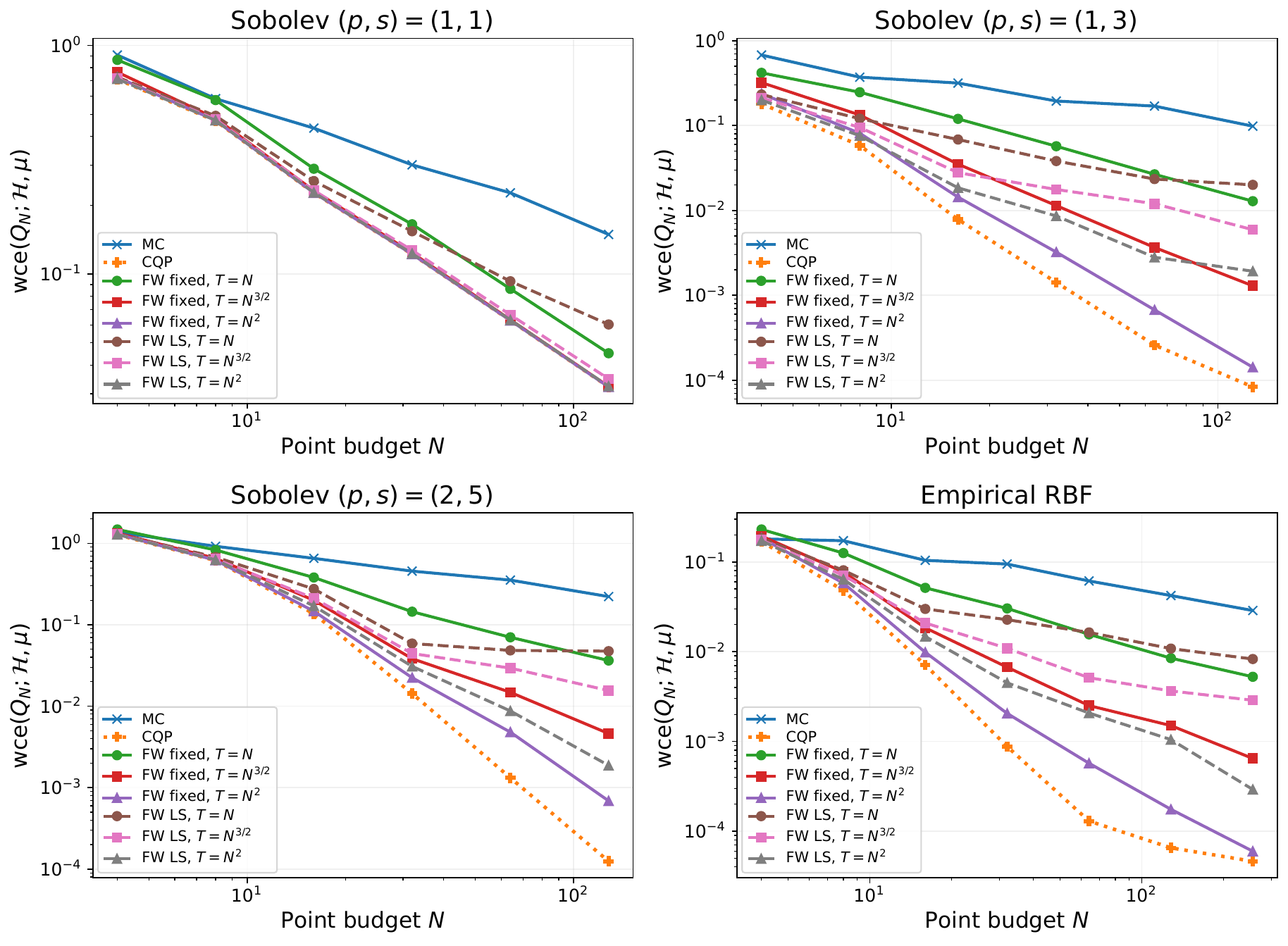}
    \caption{
    Frank--Wolfe budget and step-size ablation.
    We compare $T(N)\in\{N,N^{3/2},N^2\}$ with either the fixed step-size rule $\gamma_t=2/(t+2)$ or exact line search along the current Frank--Wolfe segment.
    Monte Carlo and fixed-pool CQP are included as references.
    }
    \label{fig:fw-ablation}
\end{figure}

Figure~\ref{fig:fw-ablation} studies how the Frank--Wolfe performance depends on the iteration budget and step-size rule.
Increasing the iteration budget improves the fixed-step Frank--Wolfe rule, and $T(N)=N^2$ is consistently the strongest or among the strongest FW choices.
This agrees with Theorem~\ref{thm:kq-fw}, where a sufficiently large number of Frank--Wolfe iterations is needed to make the optimization error negligible relative to the statistical error.

The exact line-search variant is not uniformly better in these finite-budget experiments.
Although line search minimizes the objective along the current Frank--Wolfe
segment, the resulting trajectory is not uniformly better over a finite number
of iterations. 
Empirically, the fixed step-size rule with $T(N)=N^2$ gives the most reliable approximation to the CQP among the tested FW variants.
For this reason, we use fixed-step Frank--Wolfe with $T(N)=N^2$ as the default method in the main experiments.

\subsection{Resample-herding candidate-size ablation}

\begin{figure}[h]
    \centering
    \includegraphics[width=\textwidth]{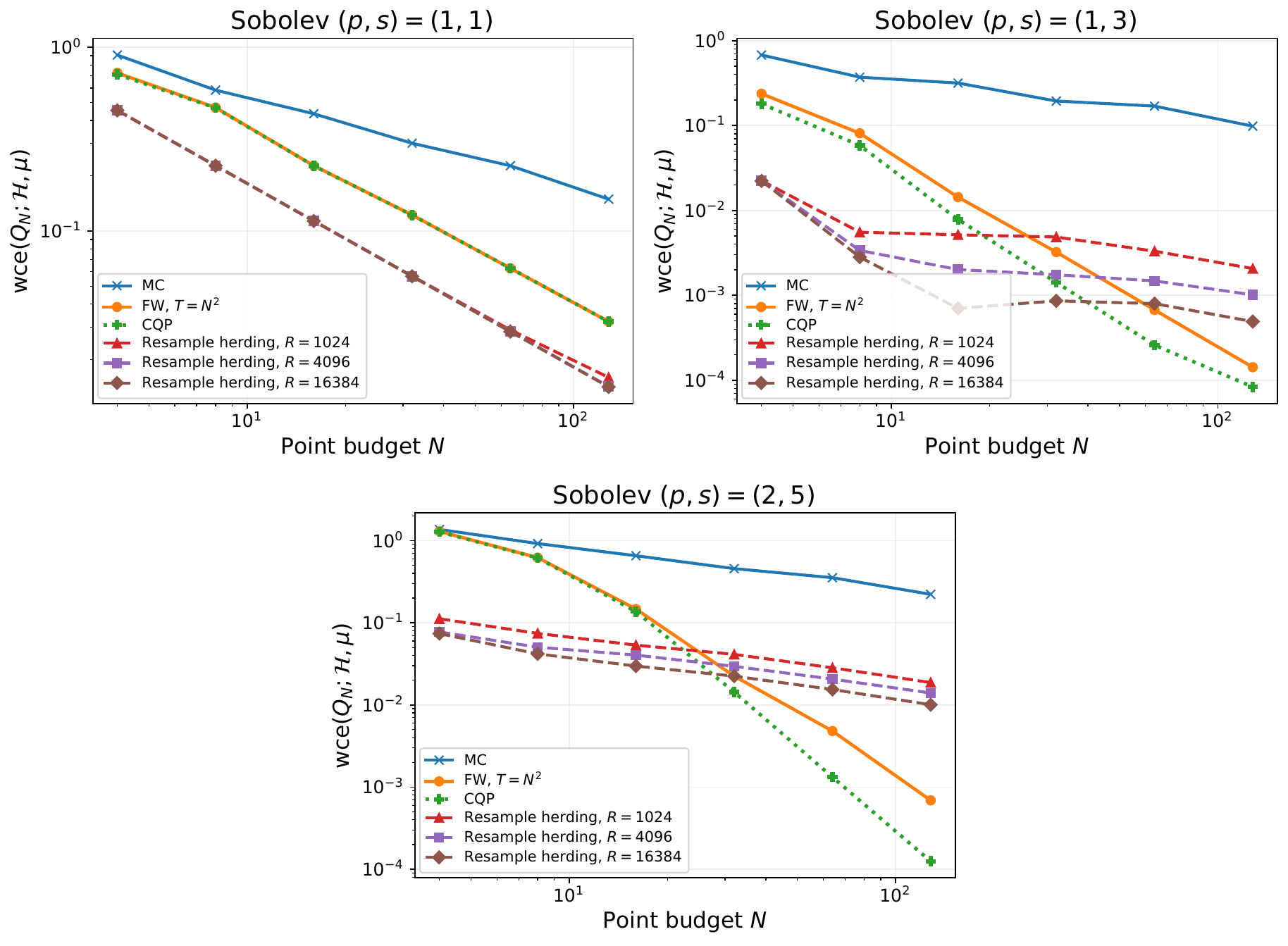}
    \caption{
    Resample-herding candidate-size ablation for the Sobolev experiments.
    We vary the number $R$ of fresh i.i.d.\ candidates searched by resample herding at each iteration, using $R\in\{1024,4096,16384\}$.
    Monte Carlo, fixed-pool Frank--Wolfe with $T(N)=N^2$, and fixed-pool CQP are included as references.
    }
    \label{fig:resample-ablation}
\end{figure}

Figure~\ref{fig:resample-ablation} examines the sensitivity of resample herding to the number $R$ of candidate points used at each iteration.
In the lower-smoothness case $(p,s)=(1,1)$, resample herding is strong and relatively insensitive to the tested candidate sizes.
This is the setting in which herding is most competitive with, and in this experiment even stronger than, the fixed-pool positive methods.

In the smoother Sobolev regimes, however, the behavior changes.
Increasing $R$ improves the resample-herding baseline to some extent, but it does not eliminate the gap to the fixed-pool positive methods at larger budgets.
The effect is especially visible in the $(p,s)=(2,5)$ experiment, where resample herding remains comparatively flat while fixed-pool FW and CQP continue to decrease.
This suggests that, in these regimes, the gap is not explained solely by the
candidate-search budget of resample herding; rather, the reweighting flexibility
over the evaluated pool appears to be an important factor.

\end{document}